\documentclass{amsproc}%\documentclass[10pt]{article}
\newtheorem{theorem}{Theorem}%[section]

\newtheorem{corollary}[theorem]{Corollary}

\newtheorem{ques}[theorem]{Question}
\theoremstyle{definition}

\theoremstyle{remark}

\numberwithin{equation}{section}
\usepackage{color,graphicx}

\newcommand{\Z}{\hbox{\bf Z}}

\def\newpic#1{%
\def\emline##1##2##3##4##5##6{%
\put(##1,##2){\special{em:point #1##3}}%
\put(##4,##5){\special{em:point #1##6}}%
\special{em:line #1##3,#1##6}}}
\newpic{}
\def\emline#1#2#3#4#5#6{%
\put(#1,#2){\special{em:moveto}}%
\put(#4,#5){\special{em:lineto}}}
\def\newpic#1{}
\title{Orienting and separating distance-transitive graphs}
\author{Italo J. Dejter}
\address{University of Puerto Rico, Rio Piedras, PR 00936-8377}
\email{ijdejter@uprrp.edu}
\date{}

\begin{document}

\begin{abstract}
It is shown that exactly 7 distance-transitive cubic graphs among
the existing 12 possess a particular ultrahomogeneous property with
respect to oriented cycles realizing the girth that allows the
construction of a related Cayley digraph with similar
ultrahomogeneous properties in which those oriented cycles appear
minimally ''pulled apart'', or ''separated'' and whose description
is truly beautiful and insightful. This work is proposed as the
initiation of a study of similar ultrahomogeneous properties for
distance-transitive graphs in general with the aim of generalizing
to constructions of similar related ''separator'' Cayley digraphs.
\end{abstract}

\maketitle

\section{Introduction}

\noindent A graph is said to be distance-transitive if its
automorphism group acts transitively on ordered pairs of vertices at
distance $i$, for each $i\ge 0$ \cite{BCN,GR,Ivanov}. In this paper
we deal mainly with finite cubic distance-transitive graphs. While
these graphs are classified and very well-understood since there are
only twelve examples, for this very restricted class of graphs we
investigate a property called ultrahomogeneity that plays a very
important role in logic, see for example \cite{Dev,Rose}. For
ultrahomogeneous graphs (resp. digraphs), we refer the reader to
\cite{Cam,Gard,GK,Ronse,Sheh} (resp. \cite{Cher,Fra,Lach}).
Distance-transitive graphs and ultrahomogeneous graphs are very
important and worthwhile to investigate. However, to start with, the
following question is answered in the affirmative for 7 of the 12
existing cubic distance-transitive graphs $G$ and negatively for the
remaining 5:

\begin{ques}
If $k$ is the largest $\ell$ such that $G$ is $\ell$-arc-transitive,
is it possible to orient all shortest cycles of $G$ so that each two
oppositely oriented $(k-1)$-arcs of $G$ are just in two
corresponding oriented shortest cycles?\end{ques}

\noindent The answer (below) to Question 1 leads to 7 connected
digraphs ${\mathcal S}(G)$ in which all oriented shortest cycles of
$G$ are minimally ''pulled apart'' or ''separated''. Specifically,
it is shown that all cubic distance-transitive graphs are
$\{C_g\}_{P_k}$-ul\-tra\-ho\-mo\-ge\-neous, where $g=$ girth, but
only the 7 cited $G$ are
$\{\vec{C_g}\}_{\vec{P}_k}$-ultrahomogeneous digraphs and in each of
these 7 digraphs $G$, the corresponding ''separator'' digraph
${\mathcal S}(G)$ is: {\bf(a)} vertex-transitive digraph of indegree
= outdegree = 2, underlying cubic graph and automorphism group  as
that of $G$; {\bf(b)} $\{\vec{C_g},\vec{C_2}\}$-ultrahomogeneous
digraph, where $\vec{C_g}=$ induced oriented $g$-cycle, with each
vertex taken as the intersection of exactly one such $\vec{C_g}$ and
one $\vec{C_2}$;  {\bf(c)} a Cayley digraph. The structure and
surface-embedding topology  \cite{BBCT,GT,White} of these ${\mathcal
S}(G)$ are studied as well. We remark that the description of these
${\mathcal S}(G)$ is truly beautiful and insightful.\bigskip

\noindent It remains to see how Question 1 can be generalized and
treated for distance-transitive graphs of degree larger than 3 and
what separator Cayley graphs could appear via such a generalization.

\section{Preliminaries}

\noindent We may consider a graph $G$ as a digraph by taking each
edge $e$ of $G$ as a pair of oppositely oriented (or O-O) arcs
$\vec{e}$ and $(\vec{e})^{-1}$ inducing an oriented 2-cycle
$\vec{C_2}$. Then, {\it fastening} $\vec{e}$ and $(\vec{e})^{-1}$
allows to obtain precisely the edge $e$ in the graph $G$. Is it
possible to orient all shortest cycles in a distance-transitive
graph $G$ so that each two O-O $(k-1)$-arcs of $G$ are in just two
oriented shortest cycles, where $k=$ largest $\ell$ such that $G$ is
$\ell$-arc transitive? It is shown below that this is so just for 7
of the 12 cubic distance-transitive graphs $G$, leading to 7
corresponding minimum connected digraphs ${\mathcal S}(G)$ in which
all oriented shortest cycles of $G$ are ''pulled apart'' by means of
a graph-theoretical operation explained in Section 4 below.\bigskip

\noindent Given a collection $\mathcal C$ of (di)graphs closed under
isomorphisms, a (di)graph $G$ is said to be $\mathcal C$-{\it
ul\-tra\-ho\-mo\-ge\-neous} (or $\mathcal C$-UH) if every
isomorphism between two induced members of $\mathcal C$ in $G$
extends to an auto\-mor\-phism of $G$. If $\mathcal C$ is the
isomorphism class of a (di)graph $H$, we say that such a $G$ is
$\{H\}$-UH or $H$-UH. In \cite{I}, $\mathcal C$-UH graphs are
defined and studied when ${\mathcal C}$ is the collection of either
the complete graphs, or the disjoint unions of complete graphs, or
the complements of those unions.\bigskip

\noindent Let $M$ be an induced subgraph of a graph $H$ and let $G$
be both an $M$-UH and an $H$-UH graph. We say that $G$ is an
$\{H\}_{M}$-{\it UH  graph} if, for each induced copy $H_0$ of $H$
in $G$ and for each induced copy $M_0$ of $M$ in $H_0$, there exists
exactly one induced copy $H_1\neq H_0$ of $H$ in $G$ with
$V(H_0)\cap V(H_1)=V(M_0)$ and $E(H_0)\cap E(H_1)=E(M_0)$. The
vertex and edge conditions above can be condensed as $H_0\cap
H_1=M_0$. We say that such a $G$ is {\it tightly fastened}. This is
generalized by saying that an $\{H\}_M$-UH graph $G$ is an
$\ell${\it -fastened} $\{H\}_M$-UH graph if given an induced copy
$H_0$ of $H$ in $G$ and an induced copy $M_0$ of $M$ in $H_0$, then
there exist exactly $\ell$ induced copies $H_i\neq H_0$ of $H$ in
$G$ such that $H_i\cap H_0\supseteq M_0$, for each
$i=1,2,\ldots,\ell$, with at least $H_1\cap H_0=M_0$.\bigskip

\noindent Let $\vec{M}$ be an induced subdigraph of a digraph
$\vec{H}$ and let the graph $G$ be both an $\vec{M}$-UH and an
$\vec{H}$-UH digraph. We say that $G$ is an
$\{\vec{H}\}_{\vec{M}}$-{\it UH digraph} if for each induced copy
$\vec{H}_0$ of $\vec{H}$ in $\vec{G}$ and for each induced copy
$\vec{M}_0$ of $\vec{M}$ in $\vec{H}_0$ there exists exactly one
induced copy $\vec{H}_1\neq\vec{H}_0$ of $\vec{H}$ in $G$ with
$V(\vec{H}_0)\cap V(\vec{H}_1)=V(\vec{M}_0)$ and
$A(\vec{H}_0)\cap\bar{A}(\vec{H}_1)=A(\vec{M}_0)$, where
$\bar{A}(\vec{H}_1)$ is formed by those arcs $(\vec{e})^{-1}$ whose
orientations are reversed with respect to the orientations of the
arcs $\vec{e}$ of $A(\vec{H}_1)$. Again, we say that such a $G$ is
{\it tightly fastened}. This case is used in the constructions of
Section 4.\bigskip

\noindent Given a finite graph $H$ and a subgraph $M$ of $H$ with
$|V(H)|>3$, we say that a graph $G$ is ({\it strongly fastened})
{\it SF} $\{H\}_M$-UH if there is a descending sequence of connected
subgraphs $M=M_1,M_2\ldots,M_t\equiv K_2$ such that: {\bf(a)}
$M_{i+1}$ is obtained from $M_i$ by the deletion of a vertex, for
$i=1,\ldots,t-1$ and {\bf(b)} $G$ is a $(2^i-1)$-fastened
$\{H\}_{M_i}$-UH graph, for $i=1,\ldots,t$.\bigskip

\noindent This paper deals with the above defined $\mathcal C$-UH
concepts applied to cubic distance-transitive (CDT) graphs
\cite{BCN}. A list of them and their main parameters
follows:\bigskip

$$\begin{array}{|l|l|l|l|l|l|l|l|l|l|}\hline
_{{\rm CDT\,\,graph }\,\, G} & _{n}  & _{d} &  _{g}  & _{k} &  _{\eta} &  _{a} & _{b} & _{h}  & _{\kappa}\\ \hline
^{{\rm Tetrahedral\,\, graph}\,\,K_4}_{{\rm{Thomsen}\,\,\rm{graph}}\,\, K_{3,3}} &  ^{4}_{6}  & ^{1}_{2} &  ^{3}_{4} &  ^{2}_{3} &  ^{4}_{9} &  ^{24}_{72} & ^{0}_{1} & ^{1}_{1} & ^{1}_{2}\\
^{{\,\,3\mbox{-}\rm cube \,\, graph}\,\,Q_3}_{\rm{Petersen}\,\,\rm{ graph}} & ^{8}_{10}  & ^{3}_{2} &  ^{4}_{5} &  ^{2}_{3} &  ^{6}_{12} &  ^{48}_{120} & ^{1}_{0} & ^{1}_{0}& ^{1}_{0} \\
^{\rm{Heawood}\,\,\rm{graph}}_{\rm{Pappus}\,\,\rm{graph}}& ^{14}_{18}  & ^{3}_{4} &  ^{6}_{6} &  ^{4}_{3} & ^{28}_{18} & ^{336}_{216} & ^{1}_{1} & ^{1}_{1} & ^{0}_{0} \\
^{\rm{Dodecahedral}\,\,\rm{graph}}_{\rm{Desargues}\,\,\rm{graph}} &  ^{20}_{20}  & ^{5}_{5} &  ^{5}_{6} &  ^{2}_{3} & ^{12}_{20} & ^{120}_{240} & ^{0}_{1} & ^{1}_{1}&^{1}_{3} \\
^{\rm{Coxeter}\,\,\rm{graph}}_{\rm{Tutte}\,\, 8\mbox{-}\rm{cage}}& ^{28}_{30}  & ^{4}_{4} &  ^{7}_{8} &  ^{3}_{5} & ^{24}_{90} & ^{336}_{1440} & ^{0}_{1} & ^{0}_{1}& ^{3}_{2} \\
^{\rm{Foster}\,\,\rm{graph}}_{\rm{Biggs}\mbox{-}\rm{Smith}\,\,\rm{graph}}& ^{90}_{102}  & ^{8}_{7} & ^{10}_{9} &  ^{5}_{4} & ^{216}_{136}&^{4320}_{2448} & ^{1}_{0} & ^{1}_{1}& ^{0}_{0}\\\hline
\end{array}$$\bigskip

\noindent where $n=$ order; $d=$ diameter; $g=$ girth; $k=$ AT or
arc-tran\-si\-ti\-vi\-ty ($=$  largest $\ell$ such that $G$ is
$\ell$-arc transitive); $\eta=$ number of $g$-cycles; $a=$ number of
automorphisms; $b$ (resp. $h$) $=1$ if $G$ is bipartite (resp.
hamiltonian) and $=0$ otherwise; and $\kappa$ is defined as follows:
let $P_k$ and $\vec{P}_k$ be respectively a $(k-1)$-path and a
directed $(k-1)$-path (of length $k-1$); let $C_g$ and $\vec{C}_g$
be respectively a cycle and a directed cycle of length $g$; then
(see Theorem 3 below): $\kappa=0$, if $G$ is not
$(\vec{C_g};\vec{P_k})$-UH; $\kappa=1$, if $G$ is planar;
$\kappa=2$, if $G$ is $\{\vec{C_g}\}_{\vec{P_k}}$-UH  with
$g=2(k-1)$; $\kappa=3$, if $G$ is $\{\vec{C_g}\}_{\vec{P_k}}$-UH
with $g>2(k-1)$.\bigskip

\noindent In Section 3 below, Theorem 2 proves that every CDT graph
is an SF $\{C_g\}_{P_k}$-UH graph, while Theorem 3 establishes
exactly which CDT graphs are not $\{\vec{C_g}\}_{\vec{P_k}}$-UH
digraphs; in fact 5 of them. Section 4 shows that each of the
remaining 7 CDT graphs $G$ yields a digraph ${\mathcal S}(G)$ whose
vertices are the $(k-1)$-arcs of $G$, an arc in ${\mathcal S}(G)$
between each two vertices representing corresponding $(k-1)$-arcs in
a common oriented $g$-cycle of $G$ and sharing just one $(k-2)$-arc;
additional arcs of ${\mathcal S}(G)$ appearing in O-O pairs
associated with the reversals of $(k-1)$-arcs of $G$. Moreover,
Theorem 4 asserts that each ${\mathcal S}(G)$ is as claimed and
itemized at end of the Introduction above.

\section{$(C_g,P_k)$-UH properties of CDT graphs}

\begin{theorem}
Let $G$ be a CDT graph of girth $=g$, AT $=k$ and order $=n$.
Then, $G$ is an SF $\{C_g\}_{P_k}$-UH graph. In particular, $G$ has exactly $2^{k-2}3ng^{-1}$ $g$-cycles.
\end{theorem}

\begin{proof}
We have to see that each CDT graph $G$ with girth $=g$ and AT $=k$
is a $(2^{i+1}-1)$-fastened $\{C_g\}_{P_{k-i}}$-UH graph, for
$i=0,1,\ldots,k-2$. In fact, each $(k-i-1)$-path $P=P_{k-i}$ of any
such $G$ is shared by exactly $2^{i+1}$ $g$-cycles of $G$, for
$i=0,1,\ldots,k-2$. For example if $k=4$, then any edge (resp.
2-path, resp. 3-path) of $G$ is shared by 8 (resp. 4, resp. 2)
$g$-cycles of $G$. This means that a $g$-cycle $C_g$ of $G$ shares a
$P_2$ (resp. $P_3$, resp. $P_4$) with exactly other 7 (resp. 3,
resp. 1) $g$-cycles. Thus $G$ is an SF $\{C_g\}_{P_{i+2}}$-UH graph,
for $i=0,1,\ldots,k-2$. The rest of the proof depends on the
particular cases analyzed in the proof of Theorem 3 below and on
some simple counting arguments for the pertaining numbers of
$g$-cycles.\end{proof}

\noindent Given a CDT graph $G$, there are just two $g$-cycles
shared by each $(k-1)$-path. If in addition $G$ is a
$\{\vec{C}_g\}_{\vec{P}_k}$-UH graph, then there exists an
assignment of an orientation for each $g$-cycle of $G$, so that the
two $g$-cycles shared by each $(k-1)$-path receive opposite
orientations. We say that such an assignment is a
$\{\vec{C}_g\}_{\vec{P}_k}$-{\it O-O assignment} (or
$\{\vec{C}_g\}_{\vec{P}_k}$-OOA). The collection of $\eta$ oriented
$g$-cycles corresponding to the $\eta$ $g$-cycles of $G$, for a
particular $\{\vec{C}_g\}_{\vec{P}_k}$-OOA will be called an
$\{\eta\vec{C}_g\}_{\vec{P}_k}$-OOC. Each such $g$-cycle will be
expressed with its successive composing vertices expressed between
parentheses but without separating commas, (as is the case for arcs
$uv$ and 2-arcs $uvw$), where as usual the vertex that succeeds the
last vertex of the cycle is its first vertex.

\begin{theorem}
The CDT graphs $G$ of girth $=g$ and AT $=k$ that are not
$\{\vec{C_g}\}_{\vec{P_k}}$-UH digraphs are the graphs of Petersen,
Heawood, Pappus, Foster and Biggs-Smith. The remaining $7$ CDT
graphs are $\{\vec{C_g}\}_{\vec{P_k}}$-UH digraphs.
\end{theorem}

\begin{proof}
Let us consider the case of each CDT graph sequentially. The graph
$K_4$ on vertex set $\{1,2,3,0\}$ admits the
$\{4\,\vec{C}_3\}_{\vec{P}_2}$-OOC $\{(123),$ $(210),(301),(032)\}.$
The graph $K_{3,3}$ obtained from $K_6$ (with vertex set
$\{1,2,3,4,5,0\}$) by deleting the edges of the triangles $(1,3,5)$
and $(2,4,0)$ admits the $\{9\,\vec{C}_4\}_{\vec{P}_3}$-OOC
$\{(1234),$ $(3210),$ $(4325),$ $(1430),$ $(2145),$ $(0125),$
$(5230),$ $(0345),$ $(5410)\}.$ The graph $Q_3$ with vertex set
$\{0,\ldots,7\}$ and edge set $\{01,$ $23,$ $45,$ $67,$ $02,$ $13,$
$46,$ $57,$ $04,$ $15,$ $26,$ $37\}$ admits the
$\{6\,\vec{C}_4\}_{\vec{P}_2}$-OOC $\{(0132)$, $(1045)$, $(3157)$,
$(2376)$, $(0264)$, $(4675)\}.$\bigskip

\noindent The Petersen graph $Pet$ is obtained from the disjoint
union of the 5-cycles $\mu^\infty=(u_0u_1u_2u_3u_4)$ and
$\nu^\infty=(v_0v_2v_4v_1v_3)$ by the addition of the edges
$(u_x,v_x)$, for $x\in\Z_5$. Apart from the two 5-cycles given
above, the other 10 5-cycles of $Pet$ can be denoted by
$\mu^x=(u_{x-1}u_xu_{x+1}v_{x+1}v_{x-1})$ and
$\nu^x=(v_{x-2}v_xv_{x+2}u_{x+2}u_{x-2})$, for each $x\in\Z_5$.
Then, the following sequence of alternating 5-cycles and 2-arcs
starts and ends up with opposite orientations:
$$\mu^2_-\,\,\,\,u_3u_2u_1\,\,\,\,\mu^\infty_+\,\,\,\,u_0u_1u_2\,\,\,\,\mu^1_-\,\,\,\,u_2v_2v_0\,\,\,\,\nu^0_-\,\,\,\,v_3u_3u_2\,\,\,\,\mu^2_+,$$
where the subindexes $\pm$ indicate either a forward or backward selection of orientation and each 2-path is presented with the orientation of the previously cited 5-cycle but must be present in the next 5-cycle with its orientation reversed.
Thus $Pet$ cannot be a $\{\vec{C}_5\}_{\vec{P}_3}$-UH digraph.\bigskip

\noindent Another way to see this is via the auxiliary table
indicated below, that presents the form in which the 5-cycles above
share the vertex sets of 2-arcs, either O-O or not. The table
details, for each one of the 5-cycles
$\xi=\mu^\infty,\nu^\infty,\mu^0,\nu^0$, (expressed as
$\xi=(\xi_0,\ldots,\xi_4)$ in the shown vertex notation), each
5-cycle $\eta$ in
$\{\mu^i,\nu^i;i=\infty,0,\ldots,4\}\setminus\{\xi\}$ that
intersects $\xi$ in the succeeding 2-paths
$\xi_i\xi_{i+1}\xi_{i+2}$, for $i=0,\ldots,4$, with additions
involving $i$ taken mod 5. Each such $\eta$ in the auxiliary table
has either a preceding minus sign, if the corresponding 2-arcs in
$\xi$ and $\eta$ are O-O, or a plus sign, otherwise. Each $-\eta_j$
(resp. $\eta_j$) shown in the table has the subindex $j$ indicating
the equality of initial vertices $\eta_j=\xi_{i+2}$ (resp.
$\eta_j=\xi_i$) of those 2-arcs, for $i=0,\ldots,4$:
$$\begin{array}{ll}
^{\mu^\infty:(+\mu^1_0,+\mu^2_0,+\mu^3_0,+\mu^4_0,+\mu^0_0),}
_{\mu^0\,\,:(+\mu^\infty_4,+\nu^3_3,-\nu^4_1,-\nu^1_4,+\nu^2_2),}&
^{\nu^\infty:(+\nu^2_0,+\nu^4_0,+\nu^1_0,+\nu^3_0,+\nu^0_0),}
_{\nu^0\,\,:(+\nu^\infty_4,-\mu^1_2,+\mu^3_4,+\mu^2_1,-\mu^4_3).}\\
\end{array}$$\
This partial auxiliary table is extended to the whole auxiliary
table by adding $x\in\Z_4$ uniformly mod 5 to all superindexes
$\ne\infty$, reconfirming that $Pet$ is not
$\{\vec{C}_5\}_{\vec{P}_3}$-UH. \bigskip

\noindent For each positive integer $n$, let $I_n$ stand for the
$n$-vertex cycle $(0,1,\ldots,n-1)$. The Heawood graph $Hea$ is
obtained from $I_{14}$ by adding the edges $(2x,5+2x)$, where
$x\in\{1,\ldots,7\}$ and operations are in $\Z_{14}$. The 28
6-cycles of $Hea$ include the following 7 6-cycles:
$$\begin{array}{l}
^{\gamma^x=  (2x,\,2x+1,\,2x+2,\,2x+3,\,2x+4,\,2x+\,\,5),\,\,\,\,\delta^x=(2x\,\,\,\,\,\,\,\,\,\,\,\,,2x+5,\,2x+6,\,2x+7,\,2x+8,\,2x+13),}
_{\epsilon^x=(2x,\,2x+5,\,2x+4,\,2x+9,\,2x+8,\,2x+13),\,\,\,\,\zeta^x=(2x+12,\,2x+3,\,2x+4,\,2x+5,\,2x\,\,\,\,\,\,\,\,\,,2x+13),}
\end{array}$$

\noindent where $x\in\Z_7$. Now, the following sequence of
alternating 6-cycles and 3-arcs starts and ends with opposite
orientations for $\gamma_0$:
$$\gamma^0_+\,\,\,2345\,\,\,\gamma^1_-\,\,\,7654\,\,\,\gamma^2_+\,\,\,6789\,\,\,\gamma^3_-\,\,\,ba98\,\,\,\gamma^4_+\,\,\,abcd\,\,\,\gamma^5_-\,\,\,10dc\,\,\,\gamma^6_+\,\,\,0123\,\,\,\gamma^0_-,$$

\noindent (where tridecimal notation is used, up to $d=13$). Thus
$Hea$ cannot be a $\{\vec{C}_7\}_{\vec{P}_4}$-UH digraph. Another
way to see this is via an auxiliary table for $Hea$ obtained in a
fashion similar to that of the one for $Pet$ above from:
$$\begin{array}{cc}
^{\gamma^0:(+\gamma^6_2,+\delta^5_1,+\gamma^1_0,+\zeta^6_1,-\epsilon^5_1,-\zeta^0_4);}_{\delta^0:(+\zeta^0_0,+\gamma^2_1,-\zeta^3_3,+\delta^4_5,+\epsilon^0_4,+\delta^3_3);}&
^{\epsilon^0:(+\epsilon^5_2,-\gamma^2_4,+\epsilon^2_0,+\zeta^4_5,+\delta^0_4,-\zeta^6_2);}_{\zeta^0:(+\delta^0_0,+\gamma^1_3,-\epsilon^1_5,-\delta^4_2,-\gamma^0_5,+\epsilon^3_3).}
\end{array}$$

\noindent This reaffirms that $Hea$ is not $\{\vec{C}_6\}_{\vec{P}_4}$-UH.\bigskip

\noindent The Pappus graph $Pap$ is obtained from $I_{18}$ by adding
to it the edges $(1+6x,6+6x),(2+6x,9+6x),(4+6x,11+6x)$, for
$x\in\{0,1,2\}$, with sums and products taken mod 18. The 6-cycles
of $Pap$ are expressible as: $A_0=(123456)$, $B_0=(3210de)$,
$C_0=(34bcde)$, $D_0=(165gh0)$, $E_0=(329ab4)$ (where octodecimal
notation is used, up to $h=17$), the 6-cycles $A_x,B_x,C_x,D_x,E_x$
obtained by uniformly adding $6x$ mod 18 to the vertices of
$A_0,B_0,C_0,D_0,E_0$, for $x\in\Z_3\setminus\{0\}$, and
$F_0=(3298fe)$, $F_1=(hg54ba)$, $F_2=(167cd0)$. No orientation
assignment makes these cycles into an
$\{18\,\vec{C}_6\}_{\vec{P}_3}$-OOC, for the following sequence of
alternating 6-cycles and 2-arcs (with orientation reversed between
each preceding 6-cycle to corresponding succeeding 6-cycle) reverses
the orientation of its initial 6-cycle in its terminal one:
$$\begin{array}{l}^{
D_1^{-1}654\,A_0123\,B_0210\,C_1h01\,D_0^{-1}g56\,C_2^{-1}876\,B_1^{-1}789
\,A_1^{-1}cba\,D_2abc\,A_1^{-1}987\,B_1^{-1}678\,C_2^{-1}765D_1}\\
^{=(654bc7)\,654\,(123456)\,123\,(3210de)\,210\,(0129ah)\,h01\,(10hg56)\,g56\,
(5gf876)\,876
}_{\;\;\;
(216789)\,789(cba987)\,cba(d0habc)\,abc(cba987)\,987\,(216789)\,678\,(5gf876)
\,765(cb4567).}
\end{array}$$

\noindent  Another way to see this is via an auxiliary table for
$Pap$ obtained in a fashion similar to those above for $Pet$ and
$Hea$, where $x=0,1,2$  (mod 3):
$$\begin{array}{l|l}
^{-A_x:(B_x,E_x\,\,\,\,\,\,,E_{x+2},D_{x+1},D_x\,\,\,\,\,\,\,\,,B_{x+1});}_{-B_x:(A_x,C_{x+1},F_2\,\,\,\,\,\,\,\,,A_{x+2},C_x\,\,\,\,\,\,\,\,,F_0\,\,\,\,\,\,\,\,);}&^{-F_0:(E_0,B_1,E_1,B_2,E_2,B_0);}_{-F_1:(D_0,E_2,D_1,E_0,D_2,E_1);}\\
^{-C_x:(E_x,D_{x+1},D_{x+2},B_{x+2},B_x\,\,\,\,\,\,\,\,,E_{x+2});}_{-D_x:(A_x,C_{x+2},F_1\,\,\,\,\,\,\,\,,A_{x+2},C_{x+1},F_2\,\,\,\,\,\,\,\,\,);}&^{-F_2:(B_1,D_1,B_2,D_2,B_0,D_0);}\\
^{-E_x:(F_0\,\,,C_{x+1},A_{x+1},F_1\,\,\,\,\,\,\,\,,C_x\,\,\,\,\,\,\,\,,A_x\,\,\,\,\,\,).}&
\end{array}$$
\noindent This reaffirms that $Pap$ is not a
$\{\vec{C}_6\}_{\vec{P}_3}$-UH digraph. In fact, observe that any
two 6-cycles here that share a 2-path possess the same orientation,
in total contrast to what happens in the 7 cases that are being
shown to be $\{\vec{C_g}\}_{\vec{P_k}}$-UH digraphs, in the course
of this proof.\bigskip

\noindent The Desargues graph $Des$ is obtained from the 20-cycle
$I_{20}$, with vertices $4x,4x+1,4x+2,4x+3$ redenoted alternatively
$x_0,x_1,x_2,x_3$, respectively, for $x\in\Z_5$, by adding the edges
$(x_3,(x+2)_0)$ and $(x_1,(x+2)_2)$, where operations are mod 5.
Then, $Des$ admits a $\{20\,\vec{C}_6\}_{\vec{P}_3}$-OOC formed by
the oriented 6-cycles $A^x,B^x,C^x,D^x$, for $x\in\{0,\ldots,4\}$,
where
$$\begin{array}{ll}
^{A^x=(x_0x_1x_2x_3(x+1)_0(x+4)_3),}_{C^x=(x_2x_1x_0(x+3)_3(x+3)_2(x+3)_1),} & ^{B^x=(x_1x_0(x+4)_3(x+4)_2(x+2)_1(x+2)_2),}_{D^x=(x_0(x+4)_3(x+1)_0(x+1)_1(x+3)_2(x+3)_3).}
\end{array}$$

\noindent The successive copies of $\vec{P}_3$ here, when reversed
in each case, must belong to the following remaining oriented
6-cycles:
$$\begin{array}{ll}
^{A^x:(C^x,C^{x+2},B^{x+1},D^{x+1},D^x\,\,\,\,\,,B^x\,\,\,\,\,\,);}_{C^x:(A^x,D^{x+4},D^x\,\,\,\,\,,A^{x+3},B^{x+1},B^{x+3});}&
^{B^x:(A^x,A^{x+4},D^{x+1},C^{x+4},C^{x+2},D^{x+4});}_{D^x:(A^x,C^{x+1},B^{x+1},B^{x+4},C^x\,\,\,\,\,\,,A^{x+4});}
\end{array}$$

\noindent showing that they constitute effectively an
$\{\eta\vec{C}_g\}_{\vec{P}_k}$-OOC.\bigskip

\noindent The dodecahedral graph $\Delta$ is a 2-covering graph of
the Petersen graph $H$, where each vertex $u_x$, (resp., $v_x$), of
$H$ is covered by two vertices $a_x,c_x$, (resp. $b_x,d_x$). A
$\{12\,\vec{C}_5\}_{\vec{P}_2}$-OOC of $\Delta$ is given by the
oriented 5-cycles $(a_0a_1a_2a_3a_4)$, $(c_4c_3c_2c_1c_0)$ and, for
each $x\in\Z_5$, also by $(a_xd_xb_{x-2}d_{x+1}a_{x+1})$ and
$(d_xb_{x+2}c_{x+2}c_{x-2}b_{x-2})$.\bigskip

\noindent The Tutte 8-cage $Tut$ is obtained from $I_{30}$, with
vertices $6x,6x+1,6x+2,6x+3,6x+4,6x+5$ denoted alternatively
$x_0,x_1,x_2,x_3,x_4,x_5$, respectively, for $x\in\Z_5$, by adding
the edges $(x_5,(x+2)_0)$, $(x_1,(x+1)_4)$ and $(x_2,(x+2)_3)$.
Then, $Tut$ admits the $\{90\,\vec{C}_8\}_{\vec{P}_5}$-OOC formed by
the oriented 8-cycles:
$$\begin{array}{ccc}
^{A^0=(4_5 0_0 0_1 0_2 0_3 0_4 0_5 1_0),}_{D^0=(3_3 3_2 3_1 4_4 4_3 4_2 1_3 1_2),} &
^{B^0=(4_2 4_3 4_4 4_5 1_0 1_1 1_2 1_3),}_{E^0=(4_5 1_0 0_5 0_4 4_1 4_0 3_5 0_0),} &
^{\,\,\,C^0=(0_2 0_3 0_4 4_1 4_0 2_5 2_4 2_3),}_{\,\,\,F^0=(4_5 0_0 3_5 4_0 2_5 2_4 1_1 1_0),} \\
^{G^0=(1_0 1_1 2_4 2_3 0_2 0_1 0_0 4_5),}_{J^0=(1_0 0_5 0_4 0_3 3_2 3_1 4_4 4_5),} &
^{H^0=(2_3 2_4 1_1 1_0 0_5 0_4 0_3 0_2),}_{K^0=(3_1 3_2 0_3 0_2 0_1 0_0 4_5 4_4),} &
^{\,\,\,I^0=(0_1 0_2 0_3 0_4 4_1 4_2 1_3 1_4),}_{\,\,L^0=(2_3 2_4 2_5 3_0 3_1 3_2 0_3 0_2),} \\
^{M^0=(3_5 4_0 4_1 0_4 0_3 0_2 0_1 0_0),}_{P^0=(4_5 4_4 4_3 4_2 4_1 0_4 0_5 1_0),} &
^{N^0=(0_0 0_1 1_4 1_5 2_0 2_1 3_4 3_5),}_{Q^0=(4_0 4_1 4_2 1_3 1_4 1_5 3_0 2_5),} &
^{O^0=(4_2 4_3 2_2 2_1 3_4 3_3 1_2 1_3),}_{R^0=(0_1 0_2 0_3 3_2 3_1 3_0 1_5 1_4),}
\end{array}$$

\noindent together with those obtained by adding $y\in\Z_5$
uniformly mod 5 to all numbers $x$ of vertices $x_i$ in
$A^0,\ldots,R^0$, for each $y=1,2,3,4$, yielding in each case
oriented 8-cycles $A^y,\ldots,R^y$.\bigskip

\noindent The Coxeter graph $Cox$ is obtained from three 7-cycles
$(u_1u_2u_3u_4u_5u_6u_0)$, $(v_4v_6v_1$ $v_3v_5v_0v_2)$, $(t_3t_6$
$t_2t_5t_1t_4t_0)$ by adding a copy of $K_{1,3}$ with degree-1
vertices $u_x,v_x,t_x$ and a central degree-3 vertex $z_x$, for each
$x\in\Z_7$. $Cox$ admits the $\{24\,\vec{C}_7\}_{\vec{P}_3}$-OOC:
$$\begin{array}{lll}\vspace*{.5mm}
^{\{{0^1}=(u_1u_2u_3u_4u_5u_6u_0),}_{\,\,\,{1^1}=(u_1z_1v_1\,v_3\,z_3u_3u_2),} &
^{{0^2}=(v_1v_3v_5v_0v_2v_4v_6),}_{{1^2}=(z_4v_4v_2v_0z_0\,t_0t_4),} &
^{{0^3}=(t_1t_5t_2\,t_6\,t_3\,t_0\,t_4),}_{{1^3}=(t_6t_2t_5z_5u_5u_6z_6),}\\
\vspace*{.5mm}
^{\,\,\,{2^1}=(v_5z_5u_5u_4u_3\,z_3\,v_3),}_{\,\,\,{3^1}=(v_5v_0z_0u_0u_6\,u_5\,z_5),} &
^{{2^2}=(t_6z_6v_6v_4v_2\,z_2t_2),}_{{3^2}=(z_4t_4t_1z_1v_1\,v_6v_4),} &
^{{2^3}=(u_1z_1t_1t_4t_0z_0u_0),}_{{3^3}=(t_6t_2z_2u_2u_3z_3t_3),} \\
\vspace*{.5mm}^{\,\,\,{4^1}=(u_1u_0z_0v_0v_2\,z_2\,u_2),}_{\,\,\,{5^1}=(z_4u_4u_3u_2z_2\,v_2\,v_4),} &
^{{4^2}=(t_6t_3z_3v_3v_1v_6\,z_6),}_{{5^2}=(v_5v_3v_1z_1t_1t_5\,z_5),} &
^{{4^3}=(z_4u_4u_5z_5t_5t_1t_4),}_{{5^3}=(t_6z_6u_6u_0z_0t_0t_3),} \\
^{\,\,\,{6^1}=(z_4v_4v_6z_6u_6\,u_5\,u_4),}_{\,\,\,{7^1}=(u_1u_0u_6z_6v_6\,v_1\,z_1),} &
^{{6^2}=(v_5v_3z_3t_3t_0\,z_0v_0),}_{{7^2}=(v_5z_5t_5t_2z_2\,v_2v_0),} &
^{{6^3}=(u_1u_2z_2t_2t_5t_1z_1),}_{{7^3}=(z_4t_4t_0t_3z_3u_3u_4)\}.}
\end{array}$$

\noindent The Foster graph $Fos$ is obtained from $I_{90}$, with
vertices $6x,6x+1,6x+2,6x+3,6x+4,6x+5$ denoted alternatively $x_0$,
$x_1$, $x_2$, $x_3$, $x_4$, $x_5$, respectively, for $x\in\Z_{15}$,
by adding the edges $(x_4,(x+2)_1)$, $(x_0,(x+2)_5)$ and
$(x_2,(x+6)_3)$. The 216 10-cycles of $Fos$ include the following 15
10-cycles, where $x\in\Z_{15}$:
$$
^{\phi^x=(x_4x_5(x+1)_0(x+1)_1(x+1)_2(x+1)_3(x+1)_4(x+1)_5(x+2)_0(x+2)_1).}
$$
Then, the following sequence of alternating 10-cycles and 4-arcs:
$$^{
\phi^0_+[1_4] \phi^1_-[3_1] \phi^2_+[3_4] \phi^3_-[5_1]
\phi^4_+[5_4] \phi^5_-[7_1] \phi^6_+[7_4] \phi^7_-[9_1]
\phi^8_+[9_4] \phi^9_-[b_1] \phi^a_+[b_4] \phi^b_-[d_1]
\phi^c_+[d_4] \phi^d_-[0_1] \phi^e_+[0_4]}$$ may be continued with
$\phi^0_-$, with opposite orientation to that of the initial
$\phi^0_+$, where $[x_j]$ stands for a 3-arc starting at the vertex
$x_j$ in the previously cited (to the left) oriented 10-cycle. Thus
$Fos$ cannot be a $\{\vec{C}_{10}\}_{\vec{P}_5}$-UH digraph. Another
way to see this is via the following table of 10-cycles of $Fos$,
where the 10-cycle $\phi_0$ intervenes as 10-cycle $0^0$:
$$\begin{array}{cc}
^{0^0=(0_4 0_5 1_0 1_1 1_2 1_3 1_4 1_5 2_0 2_1)}
_{1^0=(0_0 0_1 0_2 0_3 0_4 2_1 2_2 2_3 2_4 2_5)}&
^{8^0=(0_2 0_3 9_2 9_1 7_4 7_5 8_0 8_1 6_4 6_3)}
_{9^0=(0_4 0_5 d_0 d_1 d_2 4_3 4_4 4_5 2_0 2_1)}\\
^{2^0=(0_3 0_4 0_5 1_0 1_1 1_2 7_3 7_4 9_1 9_2)}
_{3^0=(0_3 0_4 0_5 d_0 c_5 a_0 9_5 9_4 9_3 9_2)}&
^{a^0=(0_4 0_5 d_0 d_1 b_4 b_3 b_2 2_3 2_2 2_1)}
_{b^0=(0_4 0_5 1_0 3_5 3_4 5_1 5_0 4_5 2_0 2_1)}\\
^{4^0=(0_0 0_1 0_2 6_3 6_2 6_1 6_0 5_5 3_0 2_5)}
_{5^0=(0_4 0_5 1_0 3_5 4_0 4_1 2_4 2_3 2_2 2_1)}&
^{c^0=(0_2 0_3 0_4 2_1 2_2 8_3 8_2 8_1 6_4 6_3)}
_{d^0=(0_3 0_4 2_1 2_0 4_5 5_0 7_5 7_4 9_1 9_2)}\\
^{6^0=(0_3 0_4 0_5 1_0 3_5 3_4 3_3 3_2 9_3 9_2)}
_{7^0=(0_0 0_1 R_2 0_3 9_2 9_3 3_2 3_1 3_0 2_5)}&
^{e^0=(0_5 1_0 3_5 4_0 6_5 7_0 9_5 a_0 c_5 d_0)}
_{f^0=(0_2 0_3 9_2 9_3 3_2 3_3 c_2 c_3 6_2 6_3)}
\end{array}$$

\noindent where {\bf(a)} hexadecimal notation of integers is used;
{\bf(b)} the first 14 10-cycles $x^0$, ($x=0,\ldots,13=d$), yield
corresponding 10-cycles $x^j$ , ($j\in\Z_{15}$). via translation
modulo 15 of all indexes; and {\bf(c)} the last two cycles, $e^0$
and $f^0$, yield merely  additional 10-cycles $e^1,e^2,f^1$ and
$f^2$ by the same index translation. A corresponding auxiliary table
as in the discussions for $Pet$, $Hea$ and $Pap$ above, in which the
$\pm$ assignments are missing and left as an exercise for the reader
is as follows:
$$\begin{array}{cc}
^{0^0:(2^0,4^a,1^1,1^e,3^7,2^1,8^9,b^c,b^0,8^8)}
_{1^0:(0^e,5^d,c^0,c^9,5^0,0^1,6^e,9^d,9^2,7^0)}&
^{8^0:(c^7,d^0,5^7,0^7,0^6,b^4,d^0,c^0,6^6,7^6)}
_{9^0:(1^d,4^d,2^c,2^4,3^4,1^2,a^4,b^0,b^c,a^0)}\\
^{2^0:(0^e,0^0,7^d,9^3,d^7,c^1,c^7,d^0,9^b,6^0)}
_{3^0:(6^c,d^8,e^0,5^9,6^9,0^8,6^6,c^9,6^0,9^b)}&
^{a^0:(9^b,d^b,5^9,c^b,6^8,7^2,c^9,5^0,d^8,9^0)}
_{b^0:(6^0,d^b,9^3,0^3,5^0,7^2,d^0,9^0,0^0,5^0)}\\
^{4^0:(7^c,c^d,7^6,0^5,7^3,5^2,e^e,d^d,7^0,9^2)}
_{5^0:(3^6,4^d,b^e,8^b,a^6,1^2,1^0,a^0,8^8,a^0)}&
^{c^0:(1^0,2^8,8^8,4^2,a^6,1^6,2^e,8^0,3^6,a^4)}
_{d^0:(a^4,b^0,4^2,3^7,b^4,a^7,8^0,2^0,2^8,8^9)}\\
^{6^0:(3^6,b^0,3^3,1^1,3^9,a^7,f^3,8^9,3^0,2^0)}
_{7^0:(4^9,a^d,f^3,8^9,4^3,2^2,4^c,b^d,4^0,1^0)}&
^{e^0:(3^6,4^2,3^9,4^5,3^c,4^8,3^0,4^b,3^3,4^e)}
_{f^0:(7^0,6^0,7^3,6^3,7^6,6^6,7^9,6^9,7^c,6^c)}
\end{array}$$

\noindent Let $A=(A_0$, $A_1$, $\ldots$, $A_g)$, $D=(D_0$, $D_2$,
$\ldots$, $D_f)$, $C=(C_0$, $C_4$, $\ldots$, $C_d)$, $F=(F_0$,
$F_8$, $\ldots$, $F_9)$ be 4 disjoint 17-cycles. Each $y=A,D,C,F$
has vertices $y_i$ with $i$ expressed as an heptadecimal index up to
$g=16$. We assume that $i$ is advancing in 1,2,4,8 units mod 17,
stepwise from left to right, respectively for $y=A,D,C,F$. Then the
Biggs-Smith graph $B$-$S$ is obtained by adding to the disjoint
union $A\cup D\cup C\cup F$, for each $i\in\Z_{17}$, a 6-vertex tree
$T_i$ formed by the edge-disjoint union of paths $A_iB_iC_i$,
$D_iE_iF_i$, and $B_iE_i$, where the vertices $A_i$, $D_i$, $C_i$,
$F_i$ are already present in the cycles $A$, $D$, $C$, $F$,
respectively, and where the vertices $B_i$ and $E_i$ are new and
introduced with the purpose of defining the tree $T_i$, for $0\le
i\le g=16$. Now, $\mathcal S$ has the collection ${\mathcal C}_9$ of
9-cycles formed by:
$$\begin{array}{l|l}
^{S^0=(A_0A_1B_1C_1C_5C_9C_dC_0B_0),}_{T^0=(C_0C_4B_4A_4A_3A_2A_1A_0B_0),} &
^{W^0=(A_0A_1B_1E_1F_1F_9F_0E_0B_0),}_{X^0=(C_0C_4B_4E_4D_4D_2D_0E_0B_0),} \\
^{U^0=(E_0F_0\,F_9F_1F_aF_2E_2D_2D_0,)}_{V^0=(E_0D_0D_2D_4D_6D_8E_8F_8F_0),} &
^{Y^0=(E_0B_0A_0A_1A_2B_2E_2D_2D_0),}_{Z^0=(F_0\,F_8\,E_8B_8C_8C_4C_0B_0E_0),}
\end{array}$$
and those 9-cycles obtained from these, as $S^x,\ldots,Z^x$, by
uniformly adding $x\in\Z_{17}$ mod 17 to all subindexes $i$ of
vertices $y_i$, so that $|{\mathcal C}_9|=136$.\bigskip

\noindent An auxiliary table presenting the form in which the
9-cycles above share the vertex sets of 3-arcs, either O-O or not,
is shown below, in a fashion similar to those above for $Pet$,
$Hea$, $Pap$ and $Fos$, where minus signs are set but plus signs are
tacit now:
$$\begin{array}{l|l}
 ^{S^0:(\mbox{-}T^e_1,\,\,T^1_7,\mbox{-}Z^1_4,S^d_4,\,\,S^4_3,\mbox{-}Z^9_3,\,\,T^d_0,\mbox{-}T^0_6,\,\,U^0_8),}
 _{T^0:(\,\,S^4_6,\mbox{-}S^3_0,\mbox{-}Y^2_2,\,\,T^1_4,T^g_3,\mbox{-}Y^0_1,\mbox{-}S^0_7,\,\,S^g_1,\,\,V^0_8),}
&^{W^0:(\mbox{-}U^0_4,\,W^8_2,\,W^9_1,\mbox{-}U^1_3,X^2_7,\mbox{-}X^b_4,\,Y^0_6,\mbox{-}X^0_8,\,X^9_5),}
 _{X^0:(\mbox{-}V^0_4,X^f_2,X^2_1,\mbox{-}V^4_3,\mbox{-}W^6_5,W^8_8,\mbox{-}Z^0_8,W^f_4,\mbox{-}W^0_7),}\\
^{U^0:(\,Y^g_3,\mbox{-}U^1_6,Z^1_7,\mbox{-}W^g_3,\mbox{-}W^0_0,\,Z^9_0,\mbox{-}U^g_1,\,Y^0_0,\,S^0_8),}
 _{V^0:(\mbox{-}Z^d_2,\mbox{-}V^4_6,Y^2_5,\mbox{-}X^d_3,\mbox{-}X^0_0,Y^0_7,\mbox{-}V^d_1,\mbox{-}Z^0_5,T^0_8),}
&^{Y^0:(\,\,U^0_7,\mbox{-}T^0_5,\mbox{-}T^f_2,\,\,U^1_0,\mbox{-}Y^2_8,\,\,V^f_2,\,W^0_6,\,V^0_5,\mbox{-}Y^f_4),}
 _{Z^0:(\,\,U^8_5,\mbox{-}Z^8_6,\mbox{-}V^4_0,\mbox{-}S^8_5,\mbox{-}S^g_2,\mbox{-}V^0_7,\mbox{-}Z^9_1,\,\,U^g_2,\mbox{-}X^0_6),}
\end{array}$$
This table is extended by adding $x\in\Z_{17}$ uniformly mod 17 to
all superindexes, confirming that $B$-$S$ is not
$\{\vec{C}_9\}_{\vec{P}_4}$-UH.\bigskip.
\end{proof}

\section{Separator digraphs of 7 CDT graphs}

\noindent For each of the 7 CDT graphs $G$ that are
$\{\vec{C_g}\}_{\vec{P_k}}$-UH digraphs according to Theorem 3, the
following construction yields a corresponding digraph ${\mathcal
S}(G)$ of outdegree and indegree two and having underlying cubic
graph structure and the same automorphism group of $G$. The vertices
of ${\mathcal S}(G)$ are defined as the $(k-1)$-arcs of $G$. We set
an arc in ${\mathcal S}(G)$ from each vertex $a_1a_2\ldots a_{k-1}$
into another vertex $a_2\ldots a_{k-1}a_k$ whenever there is an
oriented $g$-cycle $(a_1a_2\ldots a_{k-1}a_k\ldots)$ in the
$\{\eta\vec{C}_g\}_{\vec{P}_k}$-OOC provided by Theorem 3 to $G$.
Thus each oriented $g$-cycle in the mentioned
$\{\eta\vec{C}_g\}_{\vec{P}_k}$-OOC yields an oriented $g$-cycle of
${\mathcal S}(G)$. In addition we set an edge $e$ in ${\mathcal
S}(G)$ for each transposition of a $(k-1)$-arc of $G$, say
$h=a_1a_2\ldots a_{k-1}$, taking it into
$h^{-1}=a_{k-1}a_{k-2}\ldots a_1$. Thus the ends of $e$ are $h$ and
$h^{-1}$. As usual, the edge $e$ is considered composed by two O-O
arcs.\bigskip

\noindent The polyhedral graphs $G$ here are the tetrahedral graph
$G=K_4$, the 3-cube graph $G=Q_3$ and the dodecahedral graph
$G=\Delta$. The corresponding graphs ${\mathcal S}(G)$ have their
underlying graphs respectively being the truncated-polyhedral graphs
of the corresponding dual-polyhedral graphs that we can refer as the
truncated tetrahedron, the truncated octahedron and the truncated
icosahedron. In fact:\bigskip

\noindent{\bf(A)} ${\mathcal S}(K_4)$ has vertices
$01,02,03,12,13,23,10,20,30,21,31,32$; the  cycles $(123)$, $(210)$,
$(301)$, $(032)$ of the $\{\eta\vec{C}_g\}_{\vec{P}_k}$-OOC of $K_4$
give place to the oriented 3-cycles $(12,23$, $31)$, $(21,10,02)$,
$(30,01,13)$, $(03,32,20)$ of ${\mathcal S}(K_4)$; the additional
edges of ${\mathcal S}(K_4)$ are $(01,10)$, $(02,20)$, $(03,30)$,
$(12,21)$, $(13,31)$, $(23,32)$.\bigskip

\noindent{\bf(B)} The of oriented cycles of ${\mathcal S}(Q_3)$
corresponding to the $\{\eta\vec{C}_g\}_{\vec{P}_k}$-OOC of $Q_3$
are $(01,13,32,20)$, $(10,04,45,51)$, $(31,15,57,73)$ ,$(23,37,76,$
$62)$, $(02,26,64$, $40)$, $(46,67,75,54)$; the additional edges of
${\mathcal S}(Q_3)$ are $(01,10)$, $(23,32)$, $(45,54)$, $(67,76)$,
$(02,20)$, $(13,31)$, $(46,64)$, $(57,75)$, $(04,40)$, $(15,51)$,
$(26,62)$, $(37,73)$.\bigskip

\noindent{\bf(C)} The oriented cycles of ${\mathcal S}(\Delta)$
corresponding to the $\{\eta\vec{C}_g\}_{\vec{P}_k}$-OOC of $\Delta$
are $(a_0a_1,a_1a_2,a_2a_3,a_3a_4,a_4$ $a_0)$,
$(c_4c_3,c_3c_2,c_2c_1,c_1c_0,c_0c_4)$, and both
$(a_xd_x,d_xb_{x-2}$, $b_{x-2}d_{x+1},d_{x+1}a_{x+1},a_{x+1}a_x)$
and $(d_xb_{x+2},b_{x+2}c_{x+2},c_{x+2}c_{x-2}$,
$c_{x-2}b_{x-2},b_{x-2}d_x),$  for each $x\in\Z_5$; the additional
edges are $(a_xa_{x+1},a_{x+1}a_x)$, $(c_xc_{x+1}$, $c_{x+1}c_x)$,
$(a_xd_x,d_x$ $a_x)$, $(d_xb_{x+2},b_{x+2}d_x)$, $(b_xd_{x+2}$,
$d_{x+2}b_x)$, and $(b_xc_x,c_xb_x)$, for each $x\in\Z_5$.\bigskip

\begin{figure}[htp]%\hspace*{-9mm}
\includegraphics[scale=0.369]{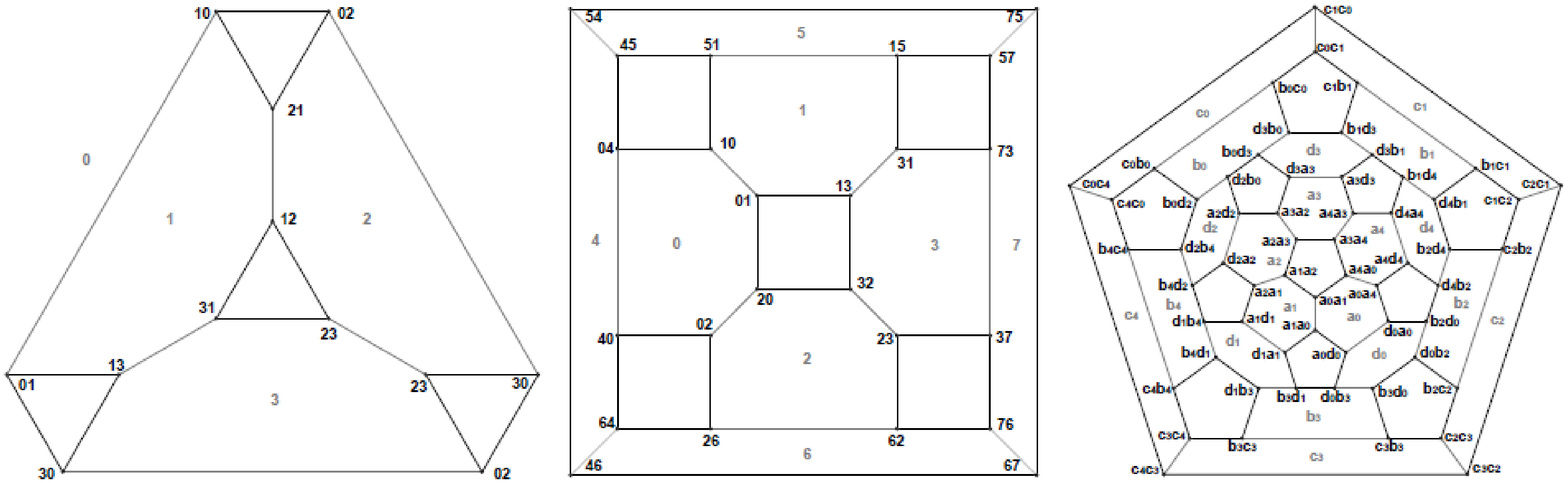}
\caption{${\mathcal S}(K_4)$, ${\mathcal S}(Q_3)$ and ${\mathcal S}(\Delta)$}
\end{figure}

\noindent Among the 7 CDT graphs $G$ that are
$\{\vec{C_g}\}_{\vec{P_k}}$-UH digraphs, the polyhedral graphs
treated above are exactly those having arc-transitivity $k=2$.
Figure 1 contains representations of these graphs ${\mathcal S}(G)$,
namely ${\mathcal S}(K_4)$, ${\mathcal S}(Q_3)$ and ${\mathcal
S}(\Delta)$, with the respective 3-cycles, 4-cycles and 5-cycles in
black trace to be considered clockwise oriented, but for the
external cycles in the cases ${\mathcal S}(Q_3)$ and ${\mathcal
S}(\Delta)$, to be considered counterclockwise oriented. The
remaining edges (to be referred as transposition edges) are gray
colored and considered bidirectional. The cycles having alternate
black and gray edges here, arising respectively from arcs from the
oriented cycles and from the transposition edges, are 6-cycles. Each
such 6-cycle has its vertices sharing the notation, indicated in
gray, of a unique vertex of the corresponding $G$. Each vertex of
$G$ is used as such gray 6-cycle indication.\bigskip

\noindent The truncated tetrahedron, truncated octahedron and
truncated icosahedron, oriented as indicated for Figure 1, are the
Cayley digraphs of the groups $A_4$, $S_4$ and $A_5$, with
respective generating sets $\{(123),(12)(34)\}$, $\{(1234),(12)\}$
and $\{(12345)$, $(23)(45)\}$. Thus

$$^{{\mathcal S}(K_4)\equiv Cay(A_4,\{(123),(12)(34)\}),{\mathcal S}(Q_3)\equiv Cay(S_4,\{(1234),(12)\}),{\mathcal S}(\Delta)\equiv Cay(A_5,\{(12345),(23)(45)\}).}$$

\noindent{\bf(D)} The oriented cycles of ${\mathcal S}(K_{3,3})$
corresponding to the $\{\eta\vec{C}_g\}_{\vec{P}_k}$-OOC of
$K_{3,3}$ are:

$$\begin{array}{ccc}
(123, 234, 341, 412),&
(321, 210, 103, 032),&
(432, 325, 254, 543),\\
(143, 430, 301, 014),&
(214, 145, 452, 521),&
(012, 125, 250, 501),\\
(523, 230, 305, 052),&
(034, 345, 450, 503),&
(541, 410, 105, 054);\end{array}$$\bigskip

\noindent and additional edges of ${\mathcal S}(K_{3,3})$ are:

$$\begin{array}{cccccc}
(123,321), (234,432), (341,143), (412,214), (210,012), (103,301),\\
(032,230), (325,523), (254,452), (543,345), (430,034), (014,410),\\
(145,541), (521,125), (250,052), (501,105), (305,503),(450,054).\end{array}$$\bigskip

\begin{figure}[htp]
  \includegraphics[scale=0.35]{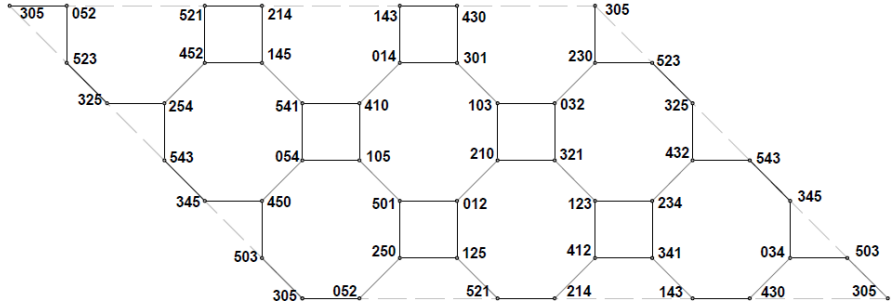}
\caption{${\mathcal S}(K_{3,3})$}
\end{figure}

\noindent A cut-out toroidal representation of ${\mathcal
S}(K_{3,3})$ is in Figure 2, where black-traced 4-cycles are
considered oriented clockwise, corresponding to the oriented
4-cycles in the $\{\eta\vec{C}_g\}_{\vec{P}_k}$-OOC of $K_{3,3}$,
and where gray edges represent transposition edges of ${\mathcal
S}(K_{3,3})$, which gives place to the alternate 8-cycles,
constituted each by an alternation of 4 transposition edges and 4
arcs of the oriented 4-cycles. In ${\mathcal S}(K_{3,3})$, there are
9 4-cycles and 9 8-cycles. Now, the group
$<(0,5,4,1)(2,3),(0,2)(1,5)>$ has order 36 and acts regularly on the
vertices of ${\mathcal S}(K_{3,3})$. For example, $(0,2)(1,5)$
stabilizes the edge $(145,541)$, and the permutation
$(0,5,4,1)(2,3)$ permutes (clockwise) the black oriented 4-cycle
$(541,410,105,054)$. Thus ${\mathcal S}(K_{3,3})$ is a Cayley
digraph. Also, observe the oriented 9-cycles in ${\mathcal
S}(K_{3,3})$ obtained by traversing alternatively 2-arcs in the
oriented 4-cycles and transposition edges; there are 6 such oriented
9-cycles.
\bigskip

\noindent {\bf(E)} The collection of oriented cycles of ${\mathcal
S}(Des)$ corresponding to the $\{\eta\vec{C}_g\}_{\vec{P}_k}$-OOC of
$Des$ is formed by the following oriented 6-cycles, where
$x\in\Z_5$:

$$\begin{array}{l}
^{(x_0x_1x_2,\,\,\,x_1x_2 x_3,\,\,\,x_2x_3x^1_0,\,\,\,x_3x^1_0x^4_3,\,\,\,x^1_0x^4_3x_0,\,\,\,x^4_3x_0x_1),}
_{(x_1x_0x^4_3,\,\,\,x_0x^4_3x^4_2,\,\,\,x^4_3x^4_2x^2_1,\,\,\,x^4_2x^2_1x^2_2,\,\,\,x^2_1x^2_2x_1,\,\,\,x^2_2x_1x_0),}\\
^{(x_2x_1x_0,\,\,\,x_1x_0x^3_3,\,\,\,x_0x^3_3x^3_2,\,\,\,x^3_3x^3_2x^3_1,\,\,\,x^3_2x^3_1x_2,\,\,\,x^3_1x_2x_1),}
_{(x_0x^4_3x^1_0,\,\,\,x^4_3x^1_0x^1_1,\,\,\,x^1_0x^1_1x^3_2,\,\,\,x^1_1x^3_2x^3_3,\,\,\,x^3_2x^3_3x_0,\,\,\,x^3_3x_0x^4_3),}
\end{array}$$\bigskip

\begin{figure}[htp]
  \includegraphics[scale=0.3]{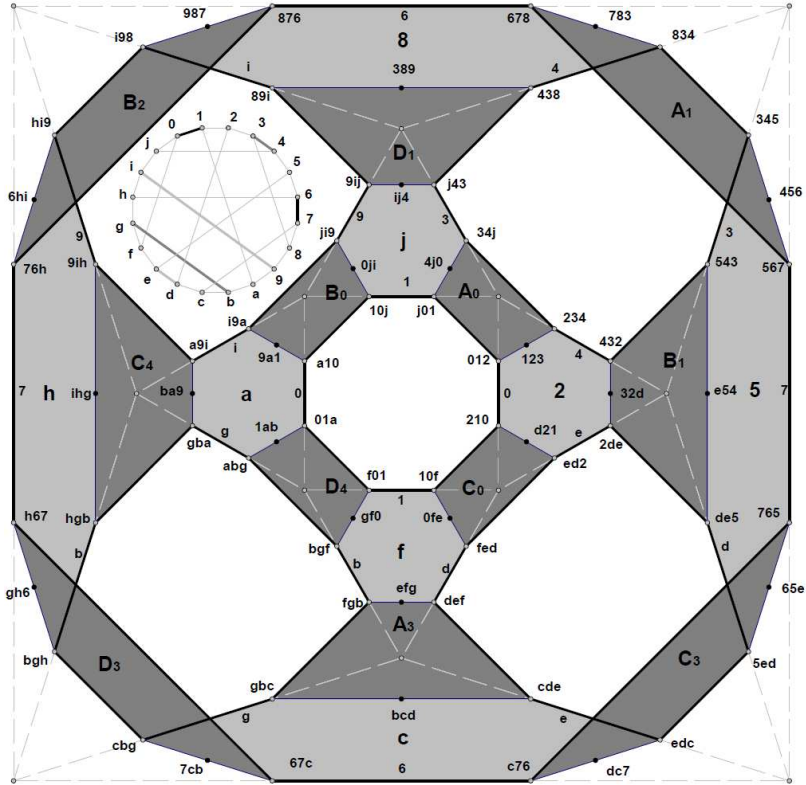}
\caption{$M_3$ and the subgraph of ${\mathcal S}(Des)$ associated to it}
\end{figure}

\noindent respectively for the 6-cycles $A^x,B^x,C^x,D^x$, where
$x^j_i$ stands for $(x+j)_i$; each of the participant vertices here
is an end of a transposition edge. Figure 3 represents a subgraph
${\mathcal S}(M_3)$ of ${\mathcal S}(Des)$ associated with the
matching $M_3$ of $Des$ indicated in its representation ``inside''
the left-upper ``eye'' of the figure, where vigesimal integer
notation is used (up to $j=19$); in the figure, additional
intermittent edges were added that form 12 square pyramids, 4 such
edges departing from a corresponding extra vertex; so, 12 extra
vertices appear that can be seen as the vertices of a cuboctahedron
whose edges are 3-paths with inner edge in ${\mathcal S}(M_3)$ and
intermittent outer edges. There is a total of 5 matchings, like
$M_3$, that we denote $M_x$, where $x=3,7,b,f,j$. In fact,
${\mathcal S}(Des)$ is obtained as the union $\bigcup\{{\mathcal
S}(M_x);x=3,7,b,f,j\}$. Observe that the components of the subgraph
induced by the matching $M_x$ in $Des$ are at mutual distance 2 and
that $M_x$ can be divided into three pairs of edges with the ends of
each pair at minimum distance 4, facts that can be used to establish
the properties of ${\mathcal S}(Des)$.\bigskip

\noindent In Figure 3 there are: 12 oriented 6-cycles (dark-gray
interiors); 6 alternate 8-cycles (thick-black edges); and 8 9-cycles
with alternate 2-arcs and transposition edges (light-gray
interiors). The 6-cycles are denoted by means of the associated
oriented 6-cycles of $Des$. Each 9-cycle has its vertices sharing
the notation of a  vertex of $V(Des)$ and this is used to denote it.
Each edge $e$ in $M_3$ has associated a closed walk in $Des$
containing every 3-path with central edge $e$; this walk can be used
to determine a unique alternate 8-cycle in ${\mathcal S}(M_3)$, and
viceversa. Each 6-cycle has two opposite (black) vertices of degree
two in ${\mathcal S}(M_3)$. In all, ${\mathcal S}(Des)$ contains 120
vertices; 360 arcs amounting to 120 arcs in oriented 6-cycles and
120 transposition edges; 20 dark-gray 6-cycles; 30 alternate
8-cycles; and 20 light-gray 9-cycles. By filling the 6-cycles and
8-cycles here with 2-dimensional faces, then the 120 vertices, 180
edges (of the underlying cubic graph) and resulting $20+30=50$ faces
yield a surface of Euler characteristic $120-180+50=-10$, so this
surface genus is 6. The automorphism group of $Des$ is
$G=S_5\times\Z_2$. Now, $G$ contains three subgroups of index 2: two
isomorphic to to $S_5$ and one isomorphic to $A_5\times\Z_2$. One of
the two subgroups of $G$ isomorphic to $S_5$ (the diagonal copy)
acts regularly on the vertices of ${\mathcal S}(Des)$ and hence
${\mathcal S}(Des)$ is a Cayley digraph.\bigskip

\noindent {\bf(F)} The collection of oriented cycles of ${\mathcal
S}(Cox)$ corresponding to the $\{\eta\vec{C}_g\}_{\vec{P}_k}$-OOC of
$Cox$ is formed by oriented 7-cycles, such as:
$$\underline{0^1}=(u_1u_2u_3,u_2u_3u_4,u_3u_4u_5,u_4u_5u_6,u_5u_6u_0,u_6u_0u_1,u_0u_1u_2),$$
and so on for the remaining oriented 7-cycles $\underline{x^y}$ with $x\in\{0,\ldots,7\}$ and $y\in\{1,2,3\}$,
based on the cor\-responding table in the proof of Theorem 3.
Moreover, each vertex of ${\mathcal S}(Cox)$ is adjacent via a transposition edge to its reversal vertex.
Thus ${\mathcal S}(Cox)$ has: underlying cubic graph; indegree $=$ outdegree $=2$; 168 vertices; 168 arcs in 24 oriented 7-cycles; 84 transposition edges; and 42 alternate 8-cycles. Its underlying cubic graph has 252 edges. From this information, by filling the 7- and 8-cycles mentioned above with 2-dimensional faces, we obtain a surface with Euler characteristic $168-252+(24+42)=-18$, so its genus is 10. On the other hand,
${\mathcal S}(Cox)$ is the Cayley digraph of the automorphism group of the Fano plane, namely $PSL(2,7)=GL(3,2)$ \cite{BL}, of order 168, with a generating set of two elements, of order 2 and 7, representable by the $3\times 3$-matrices $(100,001,010)^T$ and $(001,101,010)^T$ over the field $F_2$, where $T$ stands for transpose.\bigskip

\begin{figure}[htp]
  \includegraphics[scale=0.2]{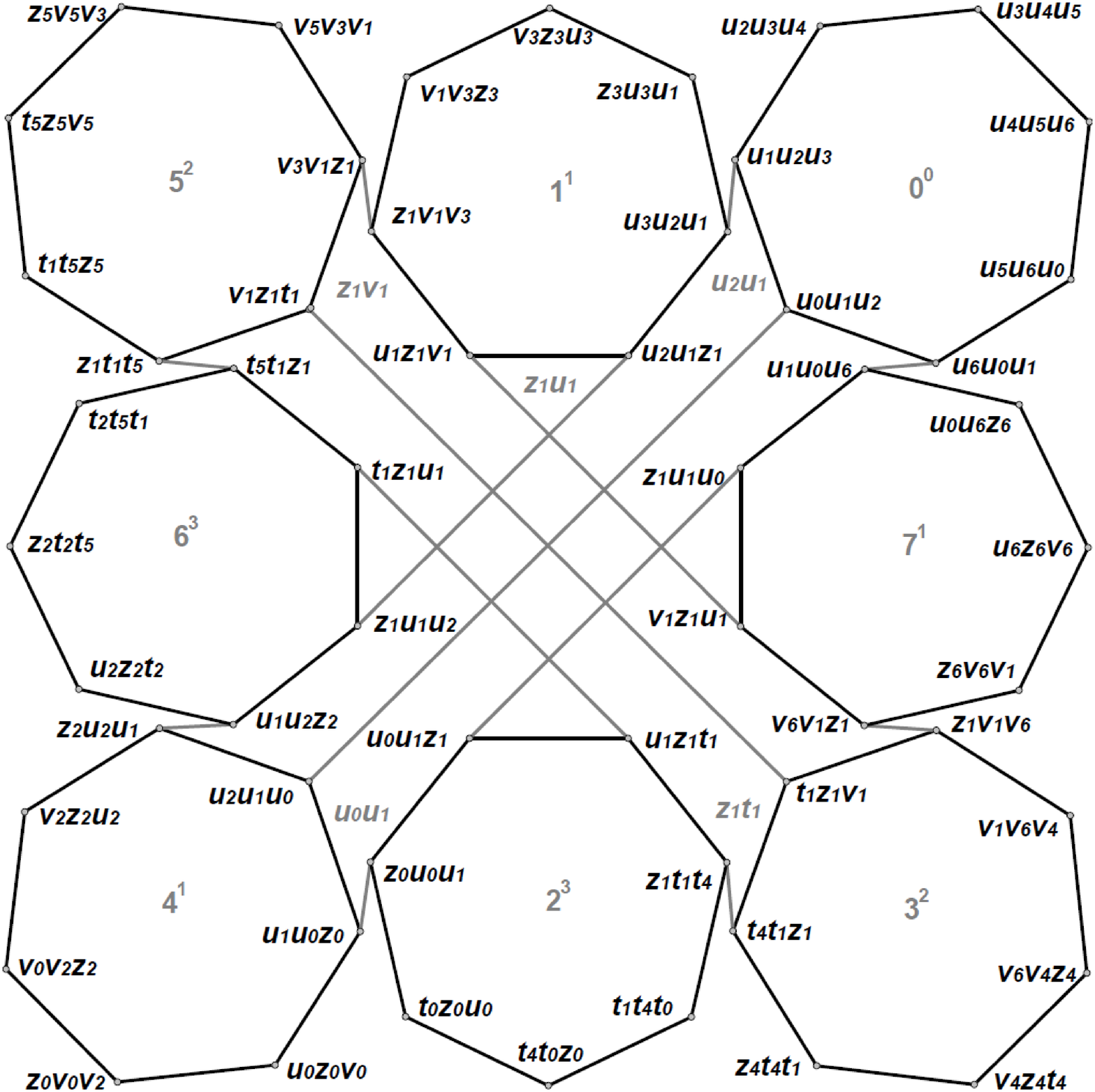}
\caption{A subdigraph of ${\mathcal S}(Cox)$ associated with an edge of $Cox$}
\end{figure}

\noindent Figure 4 depicts a subgraph of ${\mathcal S}(Cox)$
containing in its center a (twisted) alternate 8-cycle that we
denote (in gray) $z_1u_1$, and, around it, four oriented 7-cycles
adjacent to it, (namely $\underline{1^1}$, $\underline{7^1}$,
$\underline{2^3}$, $\underline{6^3}$, denoted by their corresponding
oriented 7-cycles in $Cox$, also in gray), plus four additional
oriented 7-cycles (namely $\underline{0^0}$, $\underline{3^2}$,
$\underline{4^1}$, $\underline{5^2}$), related to four 9-cycles
mentioned below. Black edges  represent arcs, and the orientation of
these 8 7-cycles is taken clockwise, with only gray edges
representing transposition edges of ${\mathcal S}(Cox)$. Each edge
of $Cox$ determines an alternate 8-cycle of ${\mathcal S}(Cox)$. In
fact, Figure 4 contains not only the alternate 8-cycle corresponding
to the edge $z_1u_1$ mentioned above, but also those corresponding
to the edges $u_1u_2,v_1z_1,t_1z_1$ and $u_0u_1$. These 8-cycles and
the 7-cycles in the figure show that alternate 8-cycles $C$ and $C'$
adjacent to a particular alternate 8-cycle $C''$ in ${\mathcal
S}(Cox)$ on opposite edges $e$ and $e'$ of $C''$ have the same
opposite edge $e''$ both to $e$ and $e'$ in $C$ and $C'$,
respectively. There are two instances of this property in Figure 4,
where the two edges taking the place of $e''$ are the large central
diagonal gray ones, with $C''$ corresponding to $u_1z_1$. As in (E)
above, the fact that each edge $e$ of $Cox$ determines an alternate
8-cycle of ${\mathcal S}(Cox)$ is related with the closed walk that
covers all the 3-paths having $e$ as central edge, and the digraph
${\mathcal S}(Cox)$ contains 9-cycles that alternate 2-arcs in the
oriented 7-cycles with transposition edges. In the case of Figure 4,
these 9-cycles are, in terms of the orientation of the 7-cycles:

$$\begin{array}{l}
^{(u_2u_1z_1,\,u_1z_1v_1,\,z_1v_1v_3,\,v_3v_1z_1,\,v_1z_1t_1,\,z_1t_1t_5,\,t_5t_1z_1,\,t_1z_1u_1,\,z_1u_1u_2),\,}
_{(v_1z_1u_1,z_1u_1u_0,u_1u_0u_6,u_6u_0u_1\!,u_0u_1u_2\!,u_1u_2u_3,u_3u_2u_1\!,u_2u_1z_1\!,u_1z_1v_1\!),}\\
^{(u_0u_1z_1,\,u_1z_1t_1,\,z_1t_1t_4,\,t_4t_1z_1,\,t_1z_1v_1,\,z_1v_1v_6,\,v_6v_1z_1,\,v_1z_1u_1,\,z_1u_1u_0),\,}
_{(t_1z_1u_1,z_1u_1u_2,u_1u_2z_2,z_2u_2u_1,u_2u_1u_0,u_1u_0z_0,z_0u_0u_1,u_0u_1z_1,u_1z_1t_1\!).}
\end{array}$$\bigskip

\noindent A convenient description of alternate 8-cycles, as those
denoted in gray in Figure 4 by the edges
$z_1u_1,z_1v_1,u_2u_1,u_0u_1,z_1t_1$ of $Cox$, is given by
indicating the successive passages through arcs of the oriented
7-cycles, with indications by means of successive subindexes in the
order of presentation of their composing vertices, which for those 5
alternate 8-cycles looks like:

$$\begin{array}{ccccc}
^{(1^1_{60},7^1_{56},2^3_{60},6^3_{56}),}&^{(5^2_{12},3^2_{23},7^1_{45},1^1_{01}),}&
^{(0^1_{01},1^1_{56},6^3_{60},4^1_{56}),}&^{(2^3_{56},7^1_{60},0^0_{56},4^1_{60}),}&
^{(3^2_{12},5^2_{23},6^3_{45},2^3_{01}).}
\end{array}$$

\noindent In a similar fashion, the four bi-alternate 9-cycles
displayed just above can be  presented by means of the shorter
expressions:

$$\begin{array}{cccc}
^{(1^1_{61},5^2_{13},6^3_{46}),}&
^{(7^1_{50},0^1_{50},1^1_{50}),}&
^{(2^3_{61},3^2_{13},7^1_{46}),}&
^{(6^3_{50},4^1_{50},2^3_{50}).}\end{array}$$

\noindent By the same token, there are twenty four tri-alternate
28-cycles, one of which is expressible as:

$$^{(0^1_{03},6^1_{03},3^2_{40},2^3_{36},5^3_{36},4^2_{62},1^1_{40}).}$$

\noindent At this point, we observe that ${\mathcal S}(K_4)$,
${\mathcal S}(Q_3)$ and ${\mathcal S}(\Delta)$ have alternate
6-cycles, while ${\mathcal S}(K_{3,3})$, ${\mathcal S}(Des)$ and
${\mathcal S}(Cox)$ have alternate 8-cycles.\bigskip

\noindent {\bf(G)} The collection of oriented cycles of ${\mathcal
S}(Tut)$ corresponding to the $\{\eta\vec{C}_g\}_{\vec{P}_k}$-OOC of
$Tut$ is formed by oriented 8-cycles, such as: $\underline{A^0}=$
$$^{(4_5 0_0 0_1 0_2 0_3,\,0_0 0_1 0_2 0_3 0_4,\,0_1 0_2 0_3 0_4
0_5,\, 0_2 0_3 0_4 0_5 1_0,\,0_3 0_4 0_5 1_0 4_5,\,0_4 0_5 1_0 4_5
0_0,\,0_5 1_0 4_5 0_0 0_1,\, 1_0 4_5 0_0 0_1 0_2)}$$ and so on for
the remaining oriented 8-cycles $\underline{X^y}$ with
$X\in\{A,\ldots,R\}$ and $y\in\Z_5$ based on the corresponding table
in the proof of Theorem 3. Moreover, each vertex of ${\mathcal
S}(Tut)$ is adjacent via a transposition edge to its reversal
vertex. Thus ${\mathcal S}(Tut)$ has: underlying cubic graph;
indegree $=$ outdegree $=2$; 720 vertices; 720 arcs in 90 oriented
8-cycles; 360 transposition edges; and 180 alternate 8-cycles, (36
of which are displayed below); its underlying cubic graph has 1080
edges. From this information, by filling the 900 8-cycles above with
2-dimensional faces, a surface with Euler characteristic
$720-1080+240=-120$ is obtained, so genus $=61$. On the other hand,
the automorphism group of $Tut$ is the projective semilinear group
$G=P\Gamma L(2,9)$ \cite{Hirschfeld}, namely the group of
collineations of the projective line $PG(1,9)$. The group $G$
contains exactly three subgroups of index 2 (and so of order 720),
one of which (namely $M_{10}$, the Mathieu group of order 10, acts
regularly on the vertices of ${\mathcal S}(Tut)$. Thus ${\mathcal
S}(Tut)$ is a Cayley digraph.
\bigskip

\noindent A fifth of the 180 alternate 8-cycles of ${\mathcal
S}(Tut)$ can be described by presenting in each case the successive
pairs of vertices in each oriented 8-cycle $\underline{X^y}$ as
follows, each such pair  denoted by means of the notation
$X^y_{u(u+1)}$ , where $u$ stands for the 4-arc in position $u$ in
$\underline{X^y}$ , with 0 indicating the first position:

$$\begin{array}{|llll|llll|llll|}
^{(A^0_{01},}_{(A^0_{34},}&^{M^0_{34},}_{J^0_{70},}&^{B^4_{34},}_{L^3_{70},}&^{K^0_{12})}_{H^0_{23})}&
^{(A^0_{12},}_{(A^0_{45},}&^{P^1_{01},}_{E^0_{70},}&^{I^0_{70},}_{P^0_{45},}&^{M^0_{23})}_{J^0_{67})}&
^{(A^0_{23},}_{(A^0_{56},}&^{H^0_{34},}_{E^1_{45},}&^{B^1_{70},}_{F^1_{70},}&^{P^1_{70})}_{E^0_{67})}\\
^{(A^0_{67},}_{(B^0_{12},}&^{G^0_{45},}_{R^3_{34},}&^{M^1_{67},}_{N^1_{56},}&^{E^1_{34})}_{C^2_{23})}&
^{(A^0_{70},}_{(B^0_{23},}&^{K^0_{23},}_{M^1_{45},}&^{L^2_{12},}_{Q^2_{67},}&^{G^0_{34})}_{R^3_{23})}&
^{(B^0_{01},}_{(B^0_{45},}&^{C^2_{34},}_{D^3_{67},}&^{Q^3_{23},}_{R^1_{70},}&^{P^0_{67})}_{K^1_{01})}\\
^{(B^0_{56},}_{(C^0_{12},}&^{D^0_{34},}_{N^4_{67},}&^{O^0_{56},}_{F^3_{45},}&^{D^3_{56})}_{G^3_{67})}&
^{(B^0_{67},}_{(C^0_{45},}&^{H^4_{45},}_{L^0_{67},}&^{C^4_{67},}_{J^2_{01},}&^{D^0_{23})}_{Q^1_{12})}&
^{(C^0_{01},}_{(C^0_{56},}&^{G^3_{70},}_{H^0_{56},}&^{H^3_{70},}_{O^4_{34},}&^{M^0_{01})}_{L^0_{56})}\\
^{(C^0_{70},}_{(D^0_{70},}&^{M^0_{12},}_{I^3_{67},}&^{I^0_{01},}_{P^4_{12},}&^{D^1_{12})}_{R^3_{67})}&
^{(D^0_{01},}_{(E^0_{01},}&^{I^4_{12},}_{N^4_{01},}&^{O^1_{12},}_{K^1_{45},}&^{I^3_{56})}_{P^0_{34})}&
^{(D^0_{45},}_{(E^0_{12},}&^{O^2_{67},}_{N^2_{34},}&^{L^1_{45},}_{F^3_{34},}&^{O^0_{45})}_{N^4_{70})}\\
^{(E^0_{23},}_{(F^0_{23},}&^{M^0_{70},}_{N^4_{45},}&^{H^3_{01},}_{R^1_{45},}&^{N^2_{23})}_{J^3_{45})}&
^{(E^0_{56},}_{(F^0_{56},}&^{F^1_{01},}_{Q^2_{56},}&^{Q^2_{45},}_{M^1_{56},}&^{F^0_{67})}_{G^0_{56})}&
^{(F^0_{12},}_{(G^0_{01},}&^{J^3_{56},}_{I^1_{45},}&^{P^3_{56},}_{O^4_{23},}&^{Q^1_{34})}_{H^0_{67})}\\
^{(G^0_{12},}_{(I^0_{23},}&^{Q^1_{01},}_{J^1_{23},}&^{J^2_{12},}_{K^1_{67},}&^{I^1_{34})}_{O^2_{01})}&
^{(G^0_{23},}_{(J^0_{34},}&^{L^2_{23},}_{R^3_{56},}&^{R^2_{12},}_{P^4_{23},}&^{Q^1_{70})}_{K^0_{56})}&
^{(H^0_{12},}_{(K^0_{70},}&^{L^3_{01},}_{R^0_{01},}&^{K^1_{34},}_{L^0_{34},}&^{N^4_{12})}_{O^1_{70})}\\
\end{array}$$\bigskip

\noindent Again, as in the previously treated cases, we may consider
the oriented paths that alternatively traverse two arcs in an
oriented 8-cycle and then a transposition edge, repeating this
operation until a closed path is formed. It happens that all such
bi-alternate cycles are 12-cycles. For example with a notation akin
to the one in the last table, we display the first row of the
corresponding table of 12-cycles:
$$\begin{array}{|cccc|cccc|cccc|}
^{(A^0_{02},}_{(\ldots\ldots,}&^{P^1_{02},}_{\ldots\ldots,}&^{R^0_{60},}_{\ldots\ldots,}&^{K^0_{02})}_{\ldots\ldots)}&
^{(A^0_{13},}_{(\ldots\ldots,}&^{H^0_{35},}_{\ldots\ldots,}&^{C^0_{60},}_{\ldots\ldots,}&^{M^0_{13})}_{\ldots\ldots)}&
^{(A^0_{24},}_{(\ldots\ldots,}&^{J^0_{71},}_{\ldots\ldots,}&^{Q^4_{13},}_{\ldots\ldots,}&^{P^1_{60})}_{\ldots\ldots)}\\
\end{array}$$
\noindent As in the case of the alternate 8-cycles above, which are
180, there are 180 bi-alternate 12-cycles in ${\mathcal S}(Tut)$. On
the other hand, an example of a tri-alternate 32-cycle in ${\mathcal
S}(Tut)$ is given by:
$$\begin{array}{cccccccc}
^{(A^0_{03},}&^{H^0_{36},}&^{O^4_{36},}&^{D^2_{50},}&^{I^0_{61},}&^{D^1_{14},}&
^{O^1_{50},}&^{K^0_{72})}.\end{array}$$ \noindent There is a total
of 90 such 32-cycles. Finally, an example of a tetra-alternate
15-cycle in ${\mathcal S}(Tut)$ is given by
$(A^0_{04},J^0_{73},K^0_{62})$, and there is a total of 240 such
15-cycles. More can be said about the relative structure of all
these types of cycles in ${\mathcal S}(Tut)$.\bigskip

\noindent The automorphism groups of the graphs ${\mathcal S}(G)$ in
items (A)-(G) above coincide with those of the corresponding graphs
$G$ because the construction of ${\mathcal S}(G)$ depends solely on
the structure of $G$ as analyzed in Section 3 above. Salient
properties of the graphs ${\mathcal S}(G)$ are contained in the
following statement.\bigskip

\begin{theorem}
For each CDT graph that is a $\{\vec{C_g}\}_{\vec{P_k}}$-UH digraph,
${\mathcal S}(G)$ is: {\bf(a)} a vertex-transitive digraph with
indegree $=$ outdegree $=2$, underlying cubic graph and the
automorphism group of $G$; {\bf(b)} a
$\{\vec{C_g},\vec{C_2}\}$-ultrahomogeneous digraph, where
$\vec{C_g}$ stands for oriented $g$-cycle coincident with its
induced subdigraph and each vertex is the intersection of exactly
one such $\vec{C_g}$ and one $\vec{C_2}$; {\bf(c)} a Cayley digraph.
Moreover, the following additional properties hold, where $s(G)=$
subjacent undirected graph of ${\mathcal S}(G)$:
\begin{enumerate}\item[(A)]
${\mathcal S}(K_4)\equiv Cay(A_4,\{(123),$ $(12)(34)\})$, $s(K_4)=$ truncated octahedron;
\item[(B)] ${\mathcal S}(Q_3)\equiv Cay(S_4,\{(1234),(12)\})$, $s(Q_3)=$ truncated octahedron;
\item[(C)] ${\mathcal S}(\Delta)\equiv Cay(A_5,$ $\{(12345),$ $(23)(45)\})$, $s(\Delta)=$ truncated icosahedron;
\item[(D)] ${\mathcal S}(K_{3,3})$ is the Cayley digraph of the subgroup of $S_6$ on the vertex set $\{0,1,2,$ $3,4,5\}$ generated by
$(0,5,4,1)(2,3)$ and $(0,2)(1,5)$ and has a toroidal embedding whose faces are delimited by $9$ oriented $4$-cycles and $9$ alternate $8$-cycles;
\item[(E)] ${\mathcal S}(Des)$ is the Cayley digraph of a diagonal copy of $S_5$ in the automorphism group $S_5\times\Z_2$ of $Des$ and has a $6$-toroidal embedding  whose faces are delimited by $20$ oriented $6$-cycles and $30$ alternate $8$-cycles;
\item[(F)] ${\mathcal S}(Cox)\equiv Cay(GL(3,2),\{(100,001,010)^T,(001,101,010)^T\})$, has a $10$-toroidal embedding whose faces are delimited by $24$ oriented $7$-cycles and $42$ alternate $8$-cycles;
\item[(G)] ${\mathcal S}(Tut)$ is the Cayley digraph of a subgroup $M_{10}$ of order 2 in the automorphism group $P\Gamma L(2,9)$ of $Tut$ and has a $61$-toroidal embedding whose faces are delimited by $90$ oriented $8$-cycles and $180$ alternate $8$-cycles.
\end{enumerate}
\end{theorem}

\begin{corollary} The bi-alternate cycles in the graphs ${\mathcal S}(G)$ above are $9$-cycles unless either $G=Q_3$ or $G=\Delta$, in which cases they are respectively $12$-cycles and $15$-cycles.
\end{corollary}

\textbf{Acknowledgement.} The author is grateful to the referee of a previous version of this paper for indications and corrections that lead to the present manuscript and for encouragement with respect to the value of the construction of the 7 directed graphs ${\mathcal S}(G)$ in Section 4.

\end{document}